\documentclass[12pt,letter]{amsart}
\usepackage{amscd}
\usepackage{amssymb}

\usepackage{amsfonts,amssymb, amscd, latexsym, graphicx, psfrag, color,float, tikz-cd}
\allowdisplaybreaks[1]

\usepackage{caption}
\usepackage{comment}
\usepackage{wrapfig}
\usepackage[all]{xy}
\usepackage{mathrsfs}
\usepackage{mathtools}
\usepackage{marvosym}
\usepackage{stmaryrd}
\usepackage{srcltx}
\usepackage{upgreek} 
\usepackage[inline]{enumitem} 

\usepackage[centering,text={15.5cm,23cm}]{geometry} 

\usepackage[bookmarks=true, bookmarksopen=true,%
bookmarksdepth=3,bookmarksopenlevel=2,%
colorlinks=true,%
linkcolor=blue,%
citecolor=blue,%
filecolor=blue,%
menucolor=blue,%
urlcolor=blue]{hyperref}

\usepackage{tikz}
\usetikzlibrary{matrix,shapes,arrows,arrows.meta,calc,topaths,intersections,hobby,positioning,decorations.pathreplacing,decorations.pathmorphing,fit,patterns}
\tikzset{cross/.style={cross out, draw=black, fill=none, minimum size=2*(#1-\pgflinewidth), inner sep=0pt, outer sep=0pt}, cross/.default={2pt}}

\usepackage{braket}

\DeclareFontFamily{U}{mathx}{}
\DeclareFontShape{U}{mathx}{m}{n}{ <-> mathx10 }{}
\DeclareSymbolFont{mathx}{U}{mathx}{m}{n}
\DeclareFontSubstitution{U}{mathx}{m}{n}
\DeclareMathAccent{\widecheck}{0}{mathx}{"71}

\definecolor{mygray}{gray}{0.75} 

\definecolor{shadecolor}{rgb}{1,0.9,0.7}

\newtheorem{theorem}{Theorem}[section]

\newtheorem{lemma-definition}[theorem]{Lemma-Definition}
\newtheorem{proposition}[theorem]{Proposition}
\newtheorem{corollary}[theorem]{Corollary}
\newtheorem{conjecture}[theorem]{Conjecture}

\theoremstyle{definition}

\newtheorem{definition}[theorem]{Definition}

\newtheorem{example}[theorem]{Example}

\newtheorem{remark}[theorem]{Remark}

\numberwithin{equation}{section}
\numberwithin{figure}{section}

\newcommand{\barM}{\overline{\mathcal{M}}}

\newcommand{\FF} {\mathbb{F}}

\newcommand{\bP}{\mathbb{P}}

\newcommand{\cO}{\mathcal{O}}

\def\mydate{\ifcase\month \or January\or February\or March\or
April\or May\or June\or July\or August\or September\or October\or 
November\or December\fi \space\number\day,\space\number\year}
\title{Open-closed duality in higher genus and winding}

\begin{document}

\author{Benjamin Zhou \textsuperscript{1}}
\email{byzhou@mail.tsinghua.edu.cn}

\address{\textsuperscript{1}Yau Mathematical Sciences Center, Tsinghua University, Haidian District,
Beijing, China}

\begin{abstract}
Let $X$ be a toric Fano surface. Let $\pi: \widehat{X} \rightarrow X$ be a toric blow up at a point. We use the Topological Vertex \cite{AKMV} to prove a higher genus, higher winding, open-closed duality between the toric Calabi-Yau 3-folds $K_X, K_{\widehat{X}}$. For $g \geq 0, w \geq 1$, we show the equality $n_g(K_{\widehat{X}}, \pi^*\beta - wC) = (-1)^g N_{g, (w)}^{LMOV}(K_X/L, \beta)$, where $n_g(K_{\widehat{X}}, \pi^*\beta - wC)$ is the genus-$g$, Gopakumar-Vafa invariant of $K_{\widehat{X}}$ in curve class $\pi^*\beta-wC$, where $\beta \in H_2(X, \mathbb{Z})$ and $C$ is the exceptional curve, and $N_{g, (w)}^{LMOV}(K_X/L, \beta)$ is the genus-$g$, LMOV invariant of an outer Aganagic-Vafa brane $L \subset K_X$ in class $\beta$ and representation $(w)$, or the Young Tableau of a single row with $w$ boxes.
\end{abstract}

\maketitle
\setcounter{tocdepth}{1}
\tableofcontents

\section{Introduction}

Let $X$ be a toric Fano surface. Let $\widehat{X} := Bl_p X$ be the toric blow up of $X$ at a toric fixed point $p \in X$, with the projection map $\pi: \widehat{X} \rightarrow X$. Let $\beta \in H^+_2(X ,\mathbb{Z})$ be an effective curve class, and let $C \in H_2(\widehat{X}, \mathbb{Z})$ be the class of the exceptional curve. Consider the toric Calabi-Yau 3-folds $K_X$, $K_{\widehat{X}}$. Let $L \subset K_X$ be an outer Aganagic-Vafa brane \cite{FL}. 

Let $g \geq 0, w \geq 1$. Let $n_g(K_{\widehat{X}}, \pi^*\beta - wC)$ be the genus-$g$, Gopakumar-Vafa invariant of $K_{\widehat{X}}$ in curve class $\pi^*\beta - wC$ \cite{GV1}\cite{GV2}. Let $N_{g, (w)}^{LMOV}(K_X/L, \beta)$ be the genus-$g$, LMOV invariant of $L \subset K_X$ in curve class $\beta \in H_2(K_X, L)$ and labelled with $U(N)$-representation $(w)$ that is a Young Tableau of a single row with $w$ boxes \cite{AKMV}. 

In this paper, we show the following equality of open and closed invariants in higher genus and winding,

\begin{theorem}
\label{thm:oc}
    For $g \geq 0, w \in \mathbb{Z}_{\geq 1}$,
    
    \[
    n_g(K_{\widehat{X}}, \pi^*\beta - wC) = (-1)^g N_{g, (w)}^{LMOV}(K_X/L, \beta)
    \]
\end{theorem}
In Section \ref{sec:GW}, the Gopakumar-Vafa invariant $n_g(K_{\widehat{X}}, \pi^*\beta - wC)$ is defined by the genus-$g$, closed Gromov-Witten invariant $N_g(K_{\widehat{X}}, \pi^*\beta - wC)$ via the Gopkaumar-Vafa formula, and the LMOV invariant $N_{g, (w)}^{LMOV}(K_X/L, \beta)$ is defined by the open Gromov-Witten invariant $O_g(K_X/L, \beta + w\beta_0 , w)$ of $L \subset K_X$ via the open multiple cover formula of \cite{MV}. The proof of Theorem \ref{thm:oc} is given in Section \ref{sec:proof}. Examples of \cite{GRZZ}, Theorem 5.3, which Theorem \ref{thm:oc} is equivalent to, are given in Section \ref{sec:examples}.

\begin{remark}
    Specialized to $w = 1$, Theorem \ref{thm:oc} is used in \cite{GRZZ} to prove a $g \geq 0$ correspondence between winding-1, open Gromov-Witten invariants of $(K_X, L)$ and 2-pointed, logarithmic Gromov-Witten invariants of $X(\log E)$, where $X(\log E)$ is the log scheme $X$ with divisorial log structure given by a smooth elliptic curve $E$. The latter, logarithmic invariants play an important role in Gross-Siebert mirror symmetry \cite{GS16} \cite{GRZZ}. For $w = 1$, Theorem \ref{thm:oc} is also used in \cite{Zho} to prove a $g \geq 0$ correspondence between winding-1, open Gromov-Witten invariants of $(K_X, L)$ and closed Gromov-Witten invariants of the projectivized canonical bundle $\bP(K_X \oplus \cO_X)$.
\end{remark}

\begin{remark}
    For $g = 0, w = 1$, the LHS of Theorem \ref{thm:oc} is the closed Gromov-Witten invariant $N_0(K_{\widehat{X}}, \pi^*\beta - C)$ defined in Section \ref{sec:GW}, since the curve class is primitive. By \cite{GRZ}, Section 2.2, the RHS of Theorem \ref{thm:oc} is the open Gromov-Witten invariant of a moment torus fiber of $K_X$ defined in \cite{FOOO}. Hence, Theorem \ref{thm:oc} generalizes \cite{LLW}, Theorem 1.1 to $g \geq 0$ and $w \geq 1$.
\end{remark}

\begin{remark}
    Numerical evidence for Theorem \ref{thm:oc} is given in \cite{GRZZ}, Appendix $A$ for $\widehat{X} = \FF_1, X = \bP^2$ in various genus, degree, and winding.
\end{remark}

\subsection{Acknowledgments}

I am grateful to Eric Zaslow for our discussions. I thank the Yau Mathematical Sciences Center, Tsinghua University and the Huiyan Talent Fund for financial support. The connection of Equation \ref{eq:before stirling} to Stirling numbers was obtained with the help of DeepSeek.

\section{Gromov-Witten invariants}

\label{sec:GW}
We define the closed and open Gromov-Witten invariants used in this paper. Let $g \geq 0, w \geq 1$. Let $\beta \in H_2(X, \mathbb{Z})$ be an effective curve class, and $C \in H_2(\widehat{X}, \mathbb{Z})$ be the class of the exceptional curve.

\subsection{Closed Gromov-Witten invariants}
Let $\barM_{g, 0}(K_{\widehat{X}}, \pi^*\beta - wC)$ be the moduli space of genus-$g$, stable maps $f: \Sigma \rightarrow K_{\widehat{X}}$ in curve class $\pi^*\beta - wC \in H_2(K_{\widehat{X}}, \mathbb{Z})$. It has virtual dimension 0. Let $[\barM_{g, 0}(K_{\widehat{X}}, \pi^*\beta - wC)]^{vir}$ be the virtual fundamental class \cite{BF}. Define the closed Gromov-Witten invariant of $K_{\widehat{X}}$,

\[
N_g(K_{\widehat{X}}, \pi^*\beta - wC) := \int_{[\barM_{g, 0}(K_{\widehat{X}}, \pi^*\beta - wC)]^{vir}} 1 \in \mathbb{Q}
\]
In string theory, Gopakumar and Vafa conjectured that the Gromov-Witten invariants of $K_{\widehat{X}}$ can be equivalently expressed in terms of integer invariants $n_g(K_{\widehat{X}}, \pi^*\beta - wC)$, in the following formula,

\begin{conjecture}[Gopakumar-Vafa conjecture, \cite{GV1}, \cite{GV2}]
\label{conj:cmc}
    \begin{equation*}
           \sum_{g, \beta} N_g(K_{\widehat{X}}, \pi^*\beta - wC) \hbar^{2g-2}Q^{\pi^*\beta - wC} = \sum_{\substack{g, \beta \\ k | w, \beta}} \frac{1}{k} n_g(K_{\widehat{X}}, \pi^*\beta - wC) \left(2 \sin \frac{k\hbar}{2}\right)^{2g-2}Q^{\pi^*\beta - wC}     
    \end{equation*}
\end{conjecture}
It is conjectured that $n_g(K_{\widehat{X}}, \pi^*\beta - wC) \in \mathbb{Z}$ and $n_g(K_{\widehat{X}}, \pi^*\beta - wC) = 0$ for $g >> 0$. The Gopakumar-Vafa conjecture has been proven for toric Calabi-Yau 3-folds \cite{Pen} \cite{Kon}. We define the Gopakumar-Vafa invariant $n_g(K_{\widehat{X}}, \pi^*\beta - wC)$ using the Gromov-Witten invariants $N_g(K_{\widehat{X}}, \pi^*\beta - wC)$ and by the above formula. 

\subsection{Open Gromov-Witten invariants}

Let $L \subset K_X$ be an outer Aganagic-Vafa brane \cite{FL}, which is a Lagrangian submanifold diffeomorphic to $S^1 \times \mathbb{R}^2$. Open Gromov-Witten invariants of $(K_X, L)$ are virtual counts of stable maps $f: (\Sigma, \partial \Sigma) \rightarrow (K_X, L)$, where $\Sigma$ is a genus-$g$, Riemann surface with boundary $\partial \Sigma$ and $f(\partial \Sigma) \subset L$. Let $w \in \mathbb{Z}_{\geq 1}$, and $\vec{k} = (k_1, \ldots, k_n) \in \mathbb{Z}_{\geq 0}^n$ be a vector of non-negative integers such that $\sum_{j=1}^n jk_j = w$. Let $|\vec{k}|$ be the number of non-zero entries of $\vec{k}$. Denote by $O_g(K_{X}/L, \beta + w\beta_0, \vec{k})$ to be the genus-$g$, winding-$w$, open Gromov-Witten invariant of an outer Aganagic-Vafa brane $L \subset K_{X}$ in framing-0, curve class $\beta+w\beta_0$ with $\beta_0 \in H_2(K_{X}, L)$, and in winding profile $\vec{k}$. The invariants $O_g(K_{X}/L, \beta+w\beta_0, \vec{k})$ can be defined using stable relative maps, and we refer to \cite{LLLZ} \cite{FL} for a detailed definition of $O_g(K_{X}/L, \beta + w\beta_0, \vec{k})$.
 
In this paper, we work with open Gromov-Witten invariants defined by $|\partial \Sigma| = 1$. Let $O_g(K_X/L, \beta + w\beta_0, w)$ be the genus-$g$, open Gromov-Witten invariant of $L \subset K_X$ in curve class $\beta + w\beta_0$, with 1 boundary component of winding $w$. Let $Q$ be a formal variable tracking the curve class $\beta$, and let $q = e^{i \hbar}$ track the genus $g$. Define the genus-$g$, generating function of $O_g(K_X/L, \beta + w\beta_0, w)$ by,

\[
F_{g, w}(Q) := \sum_{\beta} O_g(K_X/L, \beta + w\beta_0, w)Q^{\beta}
\]

As with closed Gromov-Witten invariants, open Gromov-Witten invariants are also conjectured to have a re-summation formula in terms of integer, open-BPS invariants \cite{MV}. For $n \in \mathbb{N}$, denote by $n_g^{open}(K_{X}/L, \beta + n\beta_0, n)$ to be the open-BPS invariant of $L \subset K_X$ in curve class $\beta+n\beta_0$, with 1 boundary component of winding $n$. Then, $O_g(K_{X}/L, \beta + w\beta_0, w)$ are related to the $n_g^{open}(K_{X}/L, \beta + w\beta_0, w)$ by the open multiple cover formula \cite{MV},

\begin{conjecture}[Open multiple cover formula for $F_{g,w}(Q)$, \cite{MV}]
\label{conj:omc}
\begin{multline*}
    \sum_{g = 0}^{\infty}\hbar^{2g-2+w} F_{g, w}(Q) =\\ \sum_{k | w}(-1)^{g+w} k^{w - 1} n_g^{open}\left(K_X/L, \beta+\frac{w}{k}\beta_0, \frac{w}{k}\right) \left(2 \sin \frac{k\hbar}{2}\right)^{2g-2} \left(2\sin \frac{\hbar}{2}\right)^{w}Q^{k\beta}
\end{multline*}
\end{conjecture}
An open multiple cover formula for more general winding vectors $\vec{k}$ is given in \cite{MV}. The $n^{open}_g(K_X/L, \beta+w\beta_0, w)$ are defined by the $O_g(K_X/L, \beta+w\beta_0, w)$ via the above formula. The open-BPS invariants were proven to satisfy integrality and finiteness properties \cite{Yu}.

\subsection{LMOV invariants}
Denote $N^{\mathrm{LMOV}}_{g, \mu}(K_{X}/L, \beta)$ be the genus-$g$, LMOV invariant of an outer AV-brane $L \subset K_{X}$ in degree $\beta$ and $U(N)$-representation $\mu$ \cite{AKMV}. Define its genus-$g$, degree $\beta$ generating function,

\begin{equation}
\label{eq:def_fhat}
  \widehat{f}_{\nu \varnothing \varnothing}(q, Q) := \sum_{\substack{g \geq 0, \\ \beta \in NE(X)}} N^{\mathrm{LMOV}}_{g, \mu}(K_{X}/L, \beta) (q^{\frac{1}{2}} - q^{\frac{-1}{2}})^{2g}Q^{\beta}  
\end{equation}
We will usually write $\widehat{f}_{\nu\emptyset\emptyset}$ as $\widehat{f}_{\nu}$. Explicit generating series $\widehat{f}_{\mu}$ appear in \cite{AKMV}, Section 7 for various toric Calabi-Yau 3-folds. In this paper, we consider the LMOV invariant $N^{LMOV}_{g, (w)}(K_X/L, \beta)$ in representation $(w)$, i.e. the Young Tableau that is a single row with $w$ boxes.

\subsubsection{Relation of LMOV to open-BPS invariants}
Consider a genus-$g$, open-BPS invariant $n^{open}_g(K_X/L, \beta+w\beta_0, \vec{k})$ in a winding vector $\vec{k}$ with $\sum_j jk_j = w$. Let $\mu$ be the corresponding partition that is equivalent to $\vec{k}$. The LMOV invariant $N^{\mathrm{LMOV}}_{g, \mu}(K_{X}/L, \beta)$ is related to the open-BPS invariant $n^{open}_g(K_X/L, \beta+w\beta_0, \vec{k})$ by the linear transformation, 

\[
n^{open}_g(K_X/L, \beta+w\beta_0, \vec{k}) = \sum_{|\nu| = w}\chi^{\nu}(\mu) N^{LMOV}_{g, \mu}(K_X/L, \beta)
\]
where $\chi^{\nu}(\mu)$ is the character of the symmetric group $S_w$ associated to a representation $\nu$ evaluated on $\mu$. The linear transformation $(\chi^{\nu}(\mu))_{\nu, \mu}$ is invertible over $\mathbb{Q}$ but not $\mathbb{Z}$. We define the LMOV invariants by the open-BPS invariants, or equivalently the open Gromov-Witten invariants, via the inverse transformation.

\section{Preliminaries from the Topological Vertex}

\label{sec:vertex}

We recall definitions from \cite{KM} \cite{ORV} \cite{AKMV} \cite{Mac} \cite{FH}. A partition $\mu = (\mu_1, \mu_2, \ldots)$ is a non-increasing sequence of positive integers. For an integer $k \geq 1$, the partition $k\mu$ is the non-increasing sequence $(k\mu_1, k\mu_2,\ldots)$. Define $|\mu| = \sum_i \mu_i$ and $l(\mu) = |\{i | \mu_i \neq 0\}|$. Notice $l(n\mu) = l(\mu)$ for any integer $n \geq 1$. Define $z_{\mu} := \prod_j \bar{\mu}_j! j^{\bar{\mu}_j}$, where $\bar{\mu}_j$ is the number of entries of $\mu$ of value $j$. Given an integer $n \geq 1$, we write $\mu \dashv n$ to mean a partition $\mu$ with $|\mu| = n$. Denote $\mu^T$ to be the transpose partition of $\mu$. Partitions or Young Tableaux classify representations of $U(N)$ and the symmetric group. A winding vector $\vec{k}$ defined in Section \ref{sec:GW} is equivalently a unique partition, and vice versa. Define $z_{\vec{k}} := \prod_j k_j! j^{k_j}$. We have $z_{\mu} = z_{\vec{k}}$, when $\mu$ is equivalent to $\vec{k}$, and vice versa.

Let $\Lambda \subset \mathbb{Z}[x_1, x_2, \ldots]$ be the ring of symmetric functions in an infinite number of variables $x_i$. Given a partition $\nu$, let $s_{\nu}$ be the Schur function corresponding to $\nu$. Schur functions form a $\mathbb{Z}$-basis of $\Lambda$. The Adams operations $\psi_n: \Lambda \rightarrow \Lambda$ are defined by $x_i \mapsto x_i^n$ for all $i$. For $k \geq 1$, define the $k$-th power sum $p_k := \sum_{i \geq 1} x_i^k$. The $\{p_k\}_{k \geq 1}$ form a $\mathbb{Q}$-basis of $\Lambda$. For a partition $\mu$, define $p_{\mu} := p_{\mu_1}\ldots p_{\mu_{|l(\mu)|}}$. Define $q^{\rho} := (q^{\frac{-1}{2}}, q^{\frac{-3}{2}}, q^{\frac{-5}{2}}, \ldots) = (q^{\frac{-2i+1}{2}})_{i \geq 1}$. We have $p_{\mu}(q^{\rho}) = \prod_{\mu_j \in \mu} \left(q^{\frac{\mu_j}{2}} - q^{\frac{-\mu_j}{2}}\right)^{-1}$. If $\mu$ is a partition, $q^{\rho + \mu}$ is the sequence $(q^{\frac{-2i+1}{2} + \mu_i})_{i \geq 1}$. For more detailed definitions in symmetric function theory, we refer to \cite{Mac}.

Let $S_n$ be the symmetric group of $n$ letters. Partitions of $n$ correspond to conjugacy classes of $S_n$. Let $\chi^{\nu}(\mu)$ be the character of $S_n$ associated to a $S_n$-representation $\nu$, and evaluated on a partition $\mu$. We have the relation $\chi^{\nu^T}(\mu)= (-1)^{|\mu| - l(\mu)}\chi^{\nu}(\mu)$. A formula to compute $\chi^{R}(\mu)$ is given in \cite{FH}, Formula 4.10. For a natural number $w \in \mathbb{N}$, we write $(w)$ to denote the Young Tableau that is a single row with $w$ boxes. Note that $\chi^{(w)}(\mu) \equiv 1$ for all $w \in \mathbb{N}$.

The change of basis formula between Schur functions and power sums is,

\begin{equation}
\label{eq:basis change}
\begin{split}
     s_{\nu} &= \sum_{\mu \dashv |\nu|} \frac{\chi^{\nu}(\mu)}{z_{\mu}}p_{\mu} \\
    p_{\mu} &= \sum_{\nu \dashv |\mu|}\chi^{\nu}(\mu)s_{\nu}     
\end{split}
\end{equation}

Given two Schur functions $s_{\mu}, s_{\nu}$, the Littlewood-Richardson (LR) coefficients $c_{\mu \nu}^{\lambda}$ are the structure constants of the product $s_{\mu} \cdot s_{\nu}$,

\begin{equation}
\label{eq:LR_coefficients}
 s_\mu \cdot s_{\nu} = \sum_{\lambda}c_{\mu \nu}^{\lambda}s_{\lambda} 
\end{equation}
If $c_{\mu \nu}^{\lambda} \neq 0$, then $|\lambda| = |\mu| + |\nu|$ with $\mu, \nu \subset \lambda$. LR-coefficients have an enumerative meaning as counts of lattice permutations, and satisfy $c_{\mu \nu}^{\lambda} \geq 0$. They satisfy $c_{\mu \nu}^{\lambda} = c_{\mu^T \nu^T}^{\lambda^T}$.

The Topological Vertex is an algorithm that leverages large $N$-duality to compute higher genus open and closed Gromov-Witten invariants of toric Calabi-Yau 3-folds \cite{AKMV}. We use the definition of the Topological Vertex $C_{\lambda \mu \nu}$ from \cite{KM} \cite{ORV}.

\begin{definition}[\cite{KM} \cite{ORV}]
  \label{def:vertex}
  \[
 C_{\lambda \mu  \nu} := q^{\frac{\kappa(\nu)}{2}}s_{\mu^T}(q^{\rho})\sum_{\eta} s_{\lambda/\eta}(q^{\mu^T+\rho}) s_{\nu^T/\eta}(q^{\mu + \rho})   
  \]
\end{definition}
Consider the 1-legged vertex $C_{\lambda \mu \nu}$ with $\mu = \nu = \varnothing$. Since $s_{\emptyset} = 1$, and $s_{R_1/R_2} = 0$ if $|R_2| > |R_1|$, this implies,

\[
C_{\lambda \varnothing \varnothing} = s_{\lambda}(q^{\rho})
\]
as the only contributing partition on the RHS of Definition \ref{def:vertex} will be $\eta = \emptyset$. 

We will write $C_{\lambda \emptyset \emptyset}$ as $C_{\lambda}$. By \cite{AKMV} \cite{LLLZ}, $C_{\lambda}$ is the partition  function of open string amplitudes of an outer Aganagic-Vafa brane in $\mathbb{C}^3$

\subsection{Relating \texorpdfstring{$\widehat{f}_{\nu}$}{fnu} to \texorpdfstring{$f_{\nu}$}{fnu}}

\label{sec:fnu to hatfnu}

We recall the definition of another generating function $f_{\nu}$ that is related to $\widehat{f}_{\nu}$ by linear transformation \cite{AKMV}. Let $n \geq 1$,

\[
f_{\nu}(q^{n}, Q^{n}) := \left(q^{\frac{n}{2}} - q^{\frac{-n}{2}}\right)^{-1}\sum_{\nu'} M_{\nu \nu'}(q^{n}) \widehat{f}_{\nu'}(q^{n}, Q^{n})
\]
where

\[
M_{\nu \nu'}(q^{n}) := \sum_{\nu''}C_{\nu \nu' \nu''}S_{\nu''}(q^{n})
\]
and $C_{\nu \nu' \nu''} := \sum_{\mu \dashv |\nu|} \frac{\chi^{\nu}(\mu)\chi^{\nu'}(\mu)\chi^{\nu''}(\mu)}{z_{\mu}}$ are the Clebsch-Gordon coefficients for the representations $\nu, \nu', \nu''$ \cite{FH}. In the definition of $C_{\nu \nu' \nu''}$, we have $|\nu| = |\nu'| = |\nu''|$. For a representation $\nu''$, define $S_{\nu''}(q) := (-1)^d q^{\frac{-l+1}{2} + d}$, if $\nu''$ is a hook representation (\cite{AKMV}, Section 7.3) with $l$ total boxes and $l-d$ boxes in the first row, and 0 if $\nu''$ is not a hook representation. Hence,

\begin{equation}
    \label{eq:f_by_fhat}
    f_{\nu}(q^{n}, Q^{n}) = \left(q^{\frac{n}{2}} - q^{\frac{-n}{2}}\right)^{-1}\sum_{\nu', \nu'', \mu \dashv |\nu|} \frac{\chi^{\nu}(\mu)\chi^{\nu'}(\mu)\chi^{\nu''}(\mu)}{z_{\mu}}S_{\nu''}(q^{n}) \widehat{f}_{\nu'}\left(q^{n}, Q^{n}\right)
\end{equation}
For a winding vector $\vec{k}$ with equivalent partition $\mu$, define,

\begin{align*}
    P_{\vec{k}}(q) &:= \frac{\prod_j \left(q^{\frac{-j}{2}} - q^{\frac{j}{2}}\right)^{k_j}}{q^{\frac{-1}{2}} - q^{\frac{1}{2}}} \\
    &= \frac{\prod_{\mu_j \in \mu} \left(q^{\frac{-\mu_j}{2}} - q^{\frac{\mu_j}{2}}\right)}{q^{\frac{-1}{2}} - q^{\frac{1}{2}}} =: P_{\mu}(q)
\end{align*}
Let $n \geq 1$ and $\mu$ be a partition. We have the following relation between $P_{\mu}(q^n)$ and the power sum $p_{n\mu}(q^{\rho})$. By definition, $p_{n \mu}(q^{\rho}) = \prod_{\mu_j \in \mu}\left(q^{\frac{n \mu_j}{2}} -q^{\frac{-n \mu_j}{2}}\right)^{-1}$, $P_{\mu}(q^{n}) = \left(q^{\frac{-n}{2}} - q^{\frac{n}{2}}\right)^{-1}\prod_{\mu_j \in \mu}\left(q^{\frac{-n \mu_j}{2}} - q^{\frac{n \mu_j}{2}}\right)$. Hence, we have,

\begin{equation}
\label{eq:p and P}
    (-1)^{l(\mu)} p_{n \mu}(q^{\rho})P_{\mu}(q^{n}) = \left(q^{\frac{-n}{2}} - q^{\frac{n}{2}}\right)^{-1}
\end{equation} 
We have the following relation by \cite{LMV}, Equation 3.13,

\begin{equation}
\label{eq:P to S}
    P_{\vec{k}}(q) = P_{\mu}(q) = \sum_{\nu \dashv |\mu|}\chi^{\nu}(\mu)S_{\nu}(q)
\end{equation}
Hence, Equation \ref{eq:f_by_fhat} becomes, 

\begin{equation}
    \label{eq:f_by_fhat_P}
    f_{\nu}(q^{n}, Q^{n}) = \left(q^{\frac{n}{2}} - q^{\frac{-n}{2}}\right)^{-1}\sum_{\nu', \mu \dashv |\nu|} \frac{\chi^{\nu}(\mu)\chi^{\nu'}(\mu)}{z_{\mu}}P_{\mu}(q^n) \widehat{f}_{\nu'}\left(q^{n}, Q^{n}\right)
\end{equation}

\subsection{Open string partition function}

Recall that $K_X$ is a toric Calabi-Yau 3-fold, and let $L \subset K_X$ be an outer Aganagic-Vafa brane \cite{FL}. Let $V$ be an $U(N)$-matrix with $N$ suitably large, and let $x_i$ be its eigenvalues. Define $Tr_{\nu}V$ to be the trace of $V$ in a $U(N)$-representation $\nu$. The open string partition function $Z^{open}_X(V)$ of $L \subset K_X$ can be computed by the Topological Vertex \cite{AKMV}, and is given by, 

\begin{equation}
\label{eq:Zopen}
\begin{split}
    Z^{open}_X(V) &:= \mathrm{Exp}\left(\sum_\nu f_{\nu}(q, Q) Tr_\nu V\right) \\
    &= 1 + \sum_{d=1}^{\infty}\frac{1}{d!}\left(\sum_{n=1}^{\infty}\frac{1}{n}\sum_{\nu} f_\nu(q^n, Q^n)Tr_\nu (V^n)\right)^d   
\end{split}
\end{equation}
where $Exp$ is the plethyistic exponential. We have $Tr_{\nu} V = s_{\nu}(x_i)$,  and $Tr_{\nu}(V^n) = \psi_n \circ s_{\nu}(x_i) = s_{\nu}(x_i^n)$. 

The change-of-basis formulas in Equation \ref{eq:basis change} give,

\[
Tr_{\nu}(V^n) = \psi_n \circ s_\nu = \sum_{\substack{\mu \dashv |\nu| \\ \alpha' \dashv n|\nu|}} \frac{\chi^{\nu}(\mu)\chi^{\alpha'}(n \mu) }{z_{\mu}}s_{\alpha'}
\]
Hence, for $d \geq 1$,

\[
(Tr_{\nu}(V^n))^d = \sum_{\substack{\mu_i \dashv |\nu| \\ \alpha'_i \dashv n|\nu| \\ 1 \leq i \leq d}}
\prod_{i=1}^d \left(\frac{\chi^{\nu}(\mu_i)\chi^{\alpha'_i}(n\mu_i)}{z_{\mu_i}}s_{\alpha'_i}\right)
\]
This implies $Z_X^{open}(V)$ is,

\begin{equation}
\label{eq:Zopen_expanded}   
\begin{split}
Z_X^{open}(V) = 1 + \sum_{d \geq 1}\frac{1}{d!}\sum_{\substack{n_i \\ 1 \leq i \leq d}}\sum_{\substack{\nu_i \\ 1 \leq i \leq d}}\left(\prod_{i=1}^d\frac{1}{n_i}f_{\nu_i}(q^{n_i}, y^{n_i})\right)\\
\left[\sum_{\substack{\mu_i \dashv |\nu_i| \\ 1 \leq i \leq d}} \sum_{\substack{\alpha'_i \dashv n_i |\nu_i| \\ 1 \leq i \leq d}}\prod_{i=1}^d\left(\frac{\chi^{\nu_i}(\mu_i)\chi^{\alpha'_i}(n_i\mu_i)}{z_{\mu_i}}s_{\alpha'_i}\right)\right] 
\end{split} 
\end{equation}

Since Schur functions form a basis of $\Lambda$, we write $Z^{open}_X(V) = \sum_\nu Z_\nu Tr_\nu V$ and define $Z_{\nu}$ as the coefficient of $Tr_{\nu} V$. 

\subsection{Finding \texorpdfstring{$Z_{\nu}$}{Znu}} 

Given a representation $\nu$, we find $Z_{\nu}$ in $Z_X^{open}(V)$. We find the contribution to $Z_{\nu}$ from $d =1$ and $d \geq 2$ in $Z_X^{open}(V)$. We write,

\begin{equation}
\label{eq:split_Znu}
  Z_{\nu} := Z_{\nu}^{d=1} + Z_{\nu}^{d \geq 2}  
\end{equation}

\subsubsection{\texorpdfstring{$d = 1$}{d=1}}
Restricting to $d = 1$, we have,

\begin{align*}
    Z^{open}_X(V)|_{d=1} &= 1 + \sum_{\substack{\nu_1 \\ n_1 \geq 1}}\frac{1}{n_1}f_{\nu_1}(q^{n_1}, Q^{n_1}) \left(\sum_{\substack{\mu_1 \dashv |\nu_1| \\ \alpha'_1 \dashv n_1 |\nu_1|}}\frac{\chi^{\nu_1}(\mu_1)\chi^{\alpha'_1}(n_1 \mu_1)}{z_{\mu_1}}s_{\alpha'_1}\right)
\end{align*}
By setting $\alpha'_1 = \nu$, and taking the coefficient of $s_{\nu}$, we have,

\begin{equation}
    \label{eq:Znu_from_d=1}
    Z_{\nu}^{d = 1} = \sum_{\substack{\nu_1 \\ n_1 \geq 1}}\frac{1}{n_1}f_{\nu_1}(q^{n_1}, Q^{n_1})\sum_{\mu_1 \dashv |\nu_1|}\frac{\chi^{\nu_1}(\mu_1)\chi^{\nu}(n_1 \mu_1)}{z_{\mu_1}}
\end{equation}
By Equation \ref{eq:f_by_fhat_P},

\begin{equation}
    \label{eq:Znu_d=1_fhat}
    \begin{split}
        Z_{\nu}^{d = 1} = \sum_{\substack{\nu_1 \\ n_1 \geq 1}}\left[\frac{1}{n_1\left(q^{\frac{n_1}{2}} - q^{\frac{-n_1}{2}}\right)}\left(\sum_{\substack{\nu'_1, |\nu'_1| = |\nu_1| \\ \mu_1, \tilde{\mu}_1 \dashv |\nu_1|}}\frac{\chi^{\nu_1}(\tilde{\mu}_1)\chi^{\nu'_1}(\tilde{\mu}_1)\chi^{\nu_1}(\mu_1)\chi^{\nu}(n_1 \mu_1)}{z_{\mu_1}z_{\tilde{\mu}_1}}P_{\tilde{\mu}_1}(q^{n_1})\right.\right.\\
        \left.\left.\widehat{f}_{\nu'_1}(q^{n_1}, Q^{n_1})\right)\right]
    \end{split}
\end{equation}
Switching sums, we have,

\begin{equation}
    \begin{split}
        Z_{\nu}^{d = 1} = \sum_{n_1 \geq 1}\left[\frac{1}{n_1\left(q^{\frac{n_1}{2}} - q^{\frac{-n_1}{2}}\right)}\left(\sum_{\substack{\nu'_1\\ \mu_1, \tilde{\mu}_1 \dashv |\nu'_1|}}\sum_{\nu_1, |\nu_1| = |\nu'_1|}\frac{\chi^{\nu_1}(\tilde{\mu}_1)\chi^{\nu'_1}(\tilde{\mu}_1)\chi^{\nu_1}(\mu_1)\chi^{\nu}(n_1 \mu_1)}{z_{\mu_1}z_{\tilde{\mu}_1}}\right.\right.\\
        \left.\left.P_{\tilde{\mu}_1}(q^{n_1})\widehat{f}_{\nu'_1}(q^{n_1}, Q^{n_1})\right)\right]\
    \end{split}
\end{equation}
By orthogonality of rows of the character table of the symmetric group, \\$\sum_{\nu_1} \frac{\chi^{\nu_1}(\tilde{\mu_1})\chi^{\nu_1}(\mu_1)}{z_{\mu_1}z_{\tilde{\mu_1}}} = \frac{1}{z_{\mu_1}}$ if $\mu_1 = \tilde{\mu_1}$, and 0 otherwise. Hence, the above becomes, 

\begin{equation}
\label{eq:Znu_d1_simplified}
    \begin{split}
        Z_{\nu}^{d = 1} = \sum_{n_1 \geq 1}\left[\frac{1}{n_1\left(q^{\frac{n_1}{2}} - q^{\frac{-n_1}{2}}\right)}\left(\sum_{\substack{\nu'_1\\ \mu_1 \dashv |\nu'_1|}}\frac{\chi^{\nu'_1}(\mu_1)\chi^{\nu}(n_1 \mu_1)}{z_{\mu_1}}P_{\mu_1}(q^{n_1})\widehat{f}_{\nu'_1}(q^{n_1}, Q^{n_1})\right)\right]\
    \end{split}
\end{equation}

\subsubsection{\texorpdfstring{$d \geq 2$}{d geq 2}}

When $d \geq 2$, we write,

\begin{equation}
\label{eq:prod_expression}
   s_{\alpha'_1}\ldots s_{\alpha'_d} = \sum_{\alpha'_{12}}c^{\alpha'_{12}}_{\alpha'_1 \alpha'_2}\sum_{\alpha'_{123}}c^{\alpha'_{123}}_{\alpha'_{12} \alpha'_3}\ldots \sum_{\alpha'_{12\ldots d}}c^{\alpha'_{12\ldots d}}_{\alpha'_{12\ldots d-1} \alpha'_d}s_{\alpha'_{12\ldots d}} 
\end{equation}
where $c^{\alpha'_{12\ldots d}}_{\alpha'_{12\ldots d-1} \alpha'_d}$ are the Littlewood-Richardson coefficients of the product $s_{\alpha'_{12\ldots d-1}}s_{\alpha'_d}$ for $d \geq 2$. Hence,

\begin{equation}
\label{eq:Znu_def}   
\begin{split}
    Z_X^{open}(V)|_{d \geq 2} = \sum_{d \geq 2}\frac{1}{d!}\sum_{\substack{n_i \geq 1 \\ 1 \leq i \leq d}}\sum_{\substack{\nu_i \\ 1 \leq i \leq d}} \left(\prod_{i=1}^d \frac{1}{n_i}f_{\nu_i}(q^{n_i}, Q^{n_i})\right)\\
    \left[\sum_{\substack{\mu_i \dashv |\nu_i| \\ 1 \leq i \leq d}} \sum_{\substack{\alpha'_i \dashv n_i|\nu_i| \\ 1 \leq i \leq d}}\prod_{i=1}^d\left(\frac{\chi^{\nu_i}(\mu_i)\chi^{\alpha'_i}(n_i\mu_i)}{z_{\mu_i}}\right)
    \left(\sum_{\substack{\alpha'_{12},\\ \ldots, \\ \alpha'_{12\ldots d}}}c^{\alpha'_{12}}_{\alpha'_1 \alpha'_2}\ldots c^{\alpha'_{12\ldots d}}_{\alpha'_{12\ldots d-1} \alpha'_d}s_{\alpha'_{12\ldots d}}\right) \right]
\end{split} 
\end{equation}

\begin{remark}
If the coefficient of $s_{\alpha'_{12\ldots d}}$ is non-zero, then properties of the Littlewood-Richardson coefficients require that $\alpha_i' \subset \alpha_{12\ldots d}'$ for all $i \leq d$, and $|\alpha_1'| + \ldots |\alpha_d'| = |\alpha_{12\ldots d}'|$ for each $d \geq 2$. Since $|\alpha_i'| = n_i |\nu_i|$, we have $|\alpha_{12\ldots d}'| = \sum_{i=1}^d n_i |\nu_i|$. 
\end{remark}
By setting $\alpha'_{12\ldots d} = \nu$, and taking the coefficient of $s_{\nu}$, we have,

\begin{equation}
\label{eq:Znu_dgeq2}   
\begin{split}
    Z^{d \geq 2}_{\nu} = \sum_{d \geq 2}\frac{1}{d!}\sum_{\substack{n_i \geq 1 \\ 1 \leq i \leq d}}\sum_{\substack{\nu_i \\ 1 \leq i \leq d}} \left(\prod_{i=1}^d \frac{1}{n_i}f_{\nu_i}(q^{n_i}, Q^{n_i})\right)\left[\sum_{\substack{\mu_i \dashv |\nu_i| \\ 1 \leq i \leq d}} \sum_{\substack{\alpha'_i \dashv n_i|\nu_i| \\ 1 \leq i \leq d}}\prod_{i=1}^d\left(\frac{\chi^{\nu_i}(\mu_i)\chi^{\alpha'_i}(n_i\mu_i)}{z_{\mu_i}}\right) \right.\\
    \left.\left(\sum_{\substack{\alpha'_{12},\\ \ldots, \\ \alpha'_{12\ldots d-1}}}c^{\alpha'_{12}}_{\alpha'_1 \alpha'_2}\ldots c^{\nu}_{\alpha'_{12\ldots d-1} \alpha'_d}\right)\right]
\end{split} 
\end{equation}
By Equation \ref{eq:f_by_fhat_P}, Equation \ref{eq:Znu_dgeq2} becomes

\begin{equation}
\label{eq:Znu_dgeq2_fhat_P}
\begin{split}
    Z^{d \geq 2}_{\nu} = \sum_{d \geq 2}\frac{1}{d!}\sum_{\substack{n_i \geq 1 \\ 1 \leq i \leq d}}\sum_{\substack{\nu_i \\ 1 \leq i \leq d}} \left[\sum_{\substack{\mu_i \dashv |\nu_i| \\ 1 \leq i \leq d}} \sum_{\substack{\alpha'_i \dashv n_i|\nu_i| \\ 1 \leq i \leq d}}\prod_{i=1}^d\left(\frac{\chi^{\nu_i}(\mu_i)\chi^{\alpha'_i}(n_i\mu_i)}{z_{\mu_i}}\right)\right. \\
    \left. \left(\sum_{\substack{\alpha'_{12},\\ \ldots, \\ \alpha'_{12\ldots d-1}}}c^{\alpha'_{12}}_{\alpha'_1 \alpha'_2}\ldots c^{\nu}_{\alpha'_{12\ldots d-1} \alpha'_d}\right)\right]\left(\prod_{i=1}^d \frac{1}{n_i\left(q^{\frac{n_i}{2}} - q^{\frac{-n_i}{2}}\right)}\right) \\
    \left[\sum_{\substack{\nu'_i, |\nu'_i| = |\nu_i| \\ 1 \leq i \leq d}}\left[\prod_{i=1}^d\sum_{\tilde{\mu}_i \dashv |\nu_i|}\frac{\chi^{\nu_i}(\tilde{\mu}_i)\chi^{\nu'_i}(\tilde{\mu}_i)}{z_{\tilde{\mu}_i}}P_{\tilde{\mu_i}}(q^{n_i})\right]\left(\prod_{i=1}^d\widehat{f}_{\nu'_i}(q^{n_i}, Q^{n_i})\right)\right]
\end{split}
\end{equation}
Switching the sum over $\nu'_i$ with the product over $i$, 

\begin{equation}
\label{eq:Znu_dgeq2_fhat_P_switch}
\begin{split}
    Z^{d \geq 2}_{\nu} = \sum_{d \geq 2}\frac{1}{d!}\sum_{\substack{n_i \geq 1 \\ 1 \leq i \leq d}}\left(\prod_{i=1}^d \frac{1}{n_i\left(q^{\frac{n_i}{2}} - q^{\frac{-n_i}{2}}\right)}\right) \sum_{\substack{\nu_i \\ 1 \leq i \leq d}} \\ 
    \left[\sum_{\substack{\mu_i \dashv |\nu_i| \\ 1 \leq i \leq d}} \sum_{\substack{\alpha'_i \dashv n_i|\nu_i| \\ 1 \leq i \leq d}} \prod_{i=1}^d\left(\frac{\chi^{\nu_i}(\mu_i)\chi^{\alpha'_i}(n_i\mu_i)}{z_{\mu_i}}\right)
    \left(\sum_{\substack{\alpha'_{12},\\ \ldots, \\ \alpha'_{12\ldots d-1}}}c^{\alpha'_{12}}_{\alpha'_1 \alpha'_2}\ldots c^{\nu}_{\alpha'_{12\ldots d-1} \alpha'_d}\right)\right] \\
    \left[\prod_{i=1}^d\sum_{\nu'_i, |\nu'_i| = |\nu_i|}\sum_{\tilde{\mu}_i \dashv |\nu_i|}\frac{\chi^{\nu_i}(\tilde{\mu}_i)\chi^{\nu'_i}(\tilde{\mu}_i)}{z_{\tilde{\mu}_i}}P_{\tilde{\mu}_i}(q^{n_i})\widehat{f}_{\nu'_i}(q^{n_i}, Q^{n_i})\right]
\end{split}
\end{equation}
Move the sums over $\nu'_i$ into the sums over $\mu_i$, and take $\frac{\chi^{\nu}(\mu_i)}{z_{\mu_i}}$ out. Switch the sums between $\nu_i, \nu'_i$, and $\mu_i$, under the requirement that $|\nu_i| = |\nu'_i|$, to get,

\begin{equation}
\begin{split}
    Z^{d \geq 2}_{\nu} = \sum_{d \geq 2}\frac{1}{d!}\sum_{\substack{n_i \geq 1 \\ 1 \leq i \leq d}}\left(\prod_{i=1}^d \frac{1}{n_i\left(q^{\frac{n_i}{2}} - q^{\frac{-n_i}{2}}\right)}\right)  \\ 
    \left[\sum_{\nu'_1}\sum_{\mu_1 \dashv |\nu'_1|} \sum_{\tilde{\mu}_1 \dashv |\nu'_1|}\sum_{\nu_1, |\nu_1| = |\nu'_1|}\frac{\chi^{\nu_1}(\mu_i)\chi^{\nu_1}(\tilde{\mu}_1)\chi^{\nu'_1}(\tilde{\mu}_1)}{z_{\tilde{\mu}_1}z_{\mu_1}}P_{\tilde{\mu}_1}(q^{n_1})\widehat{f}_{\nu'_1}(q^{n_1}, Q^{n_1}) \ldots \right.\\
    \sum_{\nu'_d}\sum_{\mu_d \dashv |\nu'_d|} \sum_{\tilde{\mu}_d \dashv |\nu'_d|}\sum_{\nu_d, |\nu_d| = |\nu'_d|}\frac{\chi^{\nu_d}(\mu_d)\chi^{\nu_d}(\tilde{\mu}_d)\chi^{\nu'_d}(\tilde{\mu}_d)}{z_{\tilde{\mu}_d}z_{\mu_d}}P_{\tilde{\mu}_d}(q^{n_d})\widehat{f}_{\nu'_d}(q^{n_d}, Q^{n_d}) \\
    \left.\left[\sum_{\substack{\alpha'_i \dashv n_i|\nu'_i| \\ 1 \leq i \leq d}} \left(\prod_{i=1}^d\chi^{\alpha'_i}(n_i\mu_i)\right)
    \left(\sum_{\substack{\alpha'_{12},\\ \ldots, \\ \alpha'_{12\ldots d-1}}}c^{\alpha'_{12}}_{\alpha'_1 \alpha'_2}\ldots c^{\nu}_{\alpha'_{12\ldots d-1} \alpha'_d}\right)\right]\right]
\end{split}
\end{equation}
By orthogonality of rows of the character table of the symmetric group, for all $1 \leq i \leq d$, we have $\sum_{\nu_i, |\nu_i| = |\nu'_i|} \frac{\chi^{\nu_i}(\tilde{\mu_i})\chi^{\nu_i}(\mu_i)}{z_{\mu_i}z_{\tilde{\mu_i}}} = \frac{1}{z_{\mu_i}}$ if $\mu_i = \tilde{\mu_i}$, and 0 otherwise. Hence, the above becomes, 

\begin{equation}
\label{eq:Znu_dgeq2_simplified}
\begin{split}
    Z^{d \geq 2}_{\nu} = \sum_{d \geq 2}\frac{1}{d!}\sum_{\substack{n_i \geq 1 \\ 1 \leq i \leq d}}\left(\prod_{i=1}^d \frac{1}{n_i\left(q^{\frac{n_i}{2}} - q^{\frac{-n_i}{2}}\right)}\right)  \\ 
    \left[\sum_{\nu'_1}\sum_{\mu_1 \dashv |\nu'_1|} \frac{\chi^{\nu'_1}(\mu_1)}{z_{\mu_1}}P_{\mu_1}(q^{n_1})\widehat{f}_{\nu'_1}(q^{n_1}, Q^{n_1}) \ldots \right.\\
    \sum_{\nu'_d}\sum_{\mu_d \dashv |\nu'_d|} \frac{\chi^{\nu'_d}(\mu_d)}{z_{\mu_d}}P_{\mu_d}(q^{n_d})\widehat{f}_{\nu'_d}(q^{n_d}, Q^{n_d}) \\
    \left.\left[\sum_{\substack{\alpha'_i \dashv n_i|\nu'_i| \\ 1 \leq i \leq d}} \left(\prod_{i=1}^d\chi^{\alpha'_i}(n_i\mu_i)\right)
    \left(\sum_{\substack{\alpha'_{12},\\ \ldots, \\ \alpha'_{12\ldots d-1}}}c^{\alpha'_{12}}_{\alpha'_1 \alpha'_2}\ldots c^{\nu}_{\alpha'_{12\ldots d-1} \alpha'_d}\right)\right]\right]
\end{split}
\end{equation}

\section{An equivalent open-closed statement}

Let $W$ be the flop of $K_{\widehat{X}}$ along $C$, and let $E \in H_2(W, \mathbb{Z})$ be the flopped curve. We recall, 

\begin{theorem}[\cite{GRZZ}, Theorem  5.3]
\label{thm:GRZZ}
\[
      \sum_{\substack{g \geq 0, \\ \beta \in NE(X), \\ w \in \mathbb{Z}_{>0}}}N_g(K_{\widehat{X}}, \pi^*\beta - wC)\hbar^{2g-2}Q^{\pi^*\beta - wC} = \sum_{k \geq 1}\frac{(-1)^{k+1}}{k} \left(\sum_{\nu \neq \emptyset} (-1)^{|\nu|}Z_\nu C_{\nu^T \varnothing\varnothing}Q^{|\nu|E}\right)^k  
\]
\end{theorem}
Theorem \ref{thm:GRZZ} equates the generating function of Gromov-Witten invariants of $K_{\widehat{X}}$ in curve class $\pi^*\beta - wC$ with the generating function of Gromov-Witten invariants of $W$ in curve class $\beta + wE$ by flop invariance of the Topological Vertex \cite{KM}, and applies the degeneration formula of the Topological Vertex \cite{LLLZ}. By \cite{KM}, for a fixed $w \in \mathbb{Z}_{>0}$, the coefficient of $Q^{-wC}$ is equal to the coefficient of $Q^{wE}$.

Fix a winding $w \geq 1$ in Theorem \ref{thm:GRZZ}. By the Gopakumar-Vafa formula, the LHS becomes,

\begin{equation}
\label{eq:GV}
\begin{split}
\sum_{\substack{g, \beta\\ k \geq 1\\ k| w, \beta}}\frac{(-1)^{g-1}}{k}n_g\left(K_{\widehat{X}}, \frac{\pi^*\beta - wC}{k}\right)(q^{\frac{k}{2}} - q^{\frac{-k}{2}})^{2g-2}Q^{\pi^*\beta - wC} 
\end{split}
\end{equation}
For a fixed $k | w$, recall the generating function of LMOV invariants in representation $(\frac{w}{k})$ from Section \ref{sec:GW},

\[
\widehat{f}_{(\frac{w}{k})}(q, Q) = \sum_{g, \beta}N_{g, (\frac{w}{k})}^{LMOV}\left(K_X/L, \frac{\beta}{k}\right)(q^{\frac{1}{2}} - q^{\frac{-1}{2}})^{2g-2}Q^{\frac{\beta}{k}}
\]
Cancel $Q^{-wC}$ with $Q^{wE}$ in Theorem \ref{thm:GRZZ} by flop invariance \cite{KM}. Then, substitute Theorem \ref{thm:oc} into Equation \ref{eq:GV}, which becomes,

\begin{equation}
\label{eq:GV_fhat}
 \sum_{\substack{k \geq 1\\ k| w, \beta}}\frac{-1}{k\left(q^{\frac{k}{2}} - q^{\frac{-k}{2}}\right)^{2}}\widehat{f}_{(\frac{w}{k})}(q^k, Q^k)   
\end{equation}
Expanding out the RHS of Theorem \ref{thm:GRZZ}, it becomes,

\begin{equation}
\label{eq:expanded_GRZZ}
\begin{split}
    \sum_{b=1}^w\frac{(-1)^{b+1+w}}{b}\sum_{\substack{(l_1,\ldots, l_b) \in \mathbb{Z}^b_{\geq 1} \\ \sum_i l_i = w}}\prod_{i=1}^b \left[\sum_{j_i = 1}^{p(l_i)} Z_{\nu_{j_i}}C_{\nu^T_{j_i} \varnothing \varnothing}\right]
\end{split}
\end{equation}

\begin{remark}
  In Equation \ref{eq:expanded_GRZZ}, the sum over $b$ tracks the number of representations contributing to the coefficient of $Q^{wE}$, the sum over $(l_1,\ldots, l_b) \in \mathbb{Z}^b_{\geq 1}$ tracks the length $l_i$ of each contributing representation, and the sum over $j_i$ tracks possible representations $\nu_{j_i}$ of length $l_i$.
\end{remark}
Use Equation \ref{eq:split_Znu} to write Equation \ref{eq:expanded_GRZZ} as,
    
\begin{equation}
\label{eq:oc_Znu_split}
    \sum_{b=1}^w\frac{(-1)^{b+1+w}}{b}\sum_{\substack{(l_1,\ldots, l_b) \in \mathbb{Z}^b_{\geq 1} \\ \sum_i l_i = w}}\prod_{i=1}^b \left[\sum_{j_i = 1}^{p(l_i)} \left(Z^{d=1}_{\nu_{j_i}} + Z^{d\geq 2}_{\nu_{j_i}}\right)C_{\nu^T_{j_i} \varnothing \varnothing}\right]
\end{equation}
Equating Equation \ref{eq:oc_Znu_split} with Equation \ref{eq:GV_fhat}, we have an equality of generating functions,

\begin{equation}
\label{eq:oc_split}
\begin{split}
   \sum_{\substack{k \geq 1\\ k| w, \beta}}\frac{-1}{k\left(q^{\frac{k}{2}} - q^{\frac{-k}{2}}\right)^{2}}\widehat{f}_{(\frac{w}{k})}(q^k, Q^k)  \\= \sum_{b=1}^w\frac{(-1)^{b+1+w}}{b}\sum_{\substack{(l_1,\ldots, l_b) \in \mathbb{Z}^b_{\geq 1} \\ \sum_i l_i = w}}\prod_{i=1}^b  \left[\sum_{j_i = 1}^{p(l_i)} \left(Z^{d=1}_{\nu_{j_i}} + Z^{d\geq 2}_{\nu_{j_i}}\right)C_{\nu^T_{j_i} \varnothing \varnothing}\right]
\end{split}
\end{equation}

\section{Proof of Theorem \ref{thm:oc}}
\label{sec:proof}
In this section, we show that Equation \ref{eq:oc_split} is an equality, by showing that the coefficients of individual generating functions $\widehat{f}_{R}(q^k, Q^k)$ match, and for $d \geq 2$, the coefficient of products $\prod_{i=1}^{d}\frac{1}{n_i}\widehat{f}_{\nu'_{i}}(q^{n_{i}}, Q^{n_{i}})$ appearing in the RHS of Equation \ref{eq:oc_split} vanish. The equality of generating functions in Equation \ref{eq:oc_split} implies Theorem \ref{thm:oc}, which is the equality of individual invariants.   

\subsection{Agreement of \texorpdfstring{$\widehat{f}_{R}(q^k, Q^k)$}{hat{f}R(qk, Qk)}}

\label{sec:matching}

Let $w \geq 1$ and $k | w$. We first argue that the coefficient of $\widehat{f}_{R}(q^k, Q^k)$ on the RHS of Equation \ref{eq:oc_split} for $|R| \neq \frac{w}{k}$ vanishes.

Let $\nu_{j_i}$ be a representation with $|\nu_{j_i}| = w$.  We only need to consider the $Z^{d=1}_{\nu_{j_i}}$ term in Equation \ref{eq:oc_split}. We set $\nu = \nu_{j_i}$ and $n_1 = k$ in $Z^{d=1}_{\nu_{j_i}}$ from Equation \ref{eq:Znu_d1_simplified}. In the definition of $Z^{d=1}_{\nu_{j_i}}$, we must have $|\mu_1| = \frac{w}{k}$ for the expression to be non-zero. This shows vanishing of the coefficient of $\widehat{f}_{R}(q^k, Q^k)$ for $|R| \neq \frac{w}{k}$. 

Now, suppose $|R| = \frac{w}{k}$. Set $\nu'_1 = R$ in Equation \ref{eq:Znu_d1_simplified}. Setting $b = 1$ in Equation \ref{eq:oc_split} as well as summing over all $|\nu_{j_i}| = w$,

\begin{equation}
\label{eq:exp1}
\begin{split}
    \sum_{j_1 = 1}^{p(w)} \frac{(-1)^w}{k\left(q^{\frac{k}{2}} - q^{\frac{-k}{2}}\right)}\left(\sum_{\mu_1 \dashv \frac{w}{k}}\frac{\chi^{R}(\mu_1)\chi^{\nu_{j_i}}(k \mu_1)}{z_{\mu_1}}P_{\mu_1}(q^{k})\widehat{f}_{R}(q^{k}, Q^{k})\right)\ s_{\nu^T_{j_i}}(q^{\rho})
\end{split}
\end{equation}
The change of basis from Schur functions to power sums gives, 

\begin{align*}
    \sum_{j_1 = 1}^{p(w)} \chi^{\nu_{j_i}}(k\mu_1)s_{\nu^T_{j_i}}(q^{\rho}) &= \sum_{j_1 = 1}^{p(w)} (-1)^{|k\mu_1| - l(k\mu_1)}\chi^{\nu^T_{j_i}}(k\mu_1)s_{\nu^T_{j_i}}(q^{\rho}) \\
    &= (-1)^w (-1)^{l(\mu_1)}p_{k\mu_1}(q^{\rho})
\end{align*}
Equation \ref{eq:exp1} becomes,

\begin{equation}
\label{eq:exp2}
\begin{split}
    \frac{1}{k\left(q^{\frac{k}{2}} - q^{\frac{-k}{2}}\right)}\sum_{\mu_1 \dashv \frac{w}{k}}\frac{\chi^{R}(\mu_1)}{z_{\mu_1}}P_{\mu_1}(q^{k})(-1)^{l(\mu_1)}p_{k\mu_1}(q^{\rho})\widehat{f}_{R}(q^{k}, Q^{k}) 
\end{split}
\end{equation}
By Equation \ref{eq:p and P}, Equation \ref{eq:exp2} becomes,

\begin{equation}
\label{eq:exp3}
\begin{split}
    \frac{-1}{k\left(q^{\frac{k}{2}} - q^{\frac{-k}{2}}\right)^2}\sum_{\mu_1 \dashv \frac{w}{k}}\frac{\chi^{R}(\mu_1)}{z_{\mu_1}}\widehat{f}_{R}(q^{k}, Q^{k}) 
\end{split}
\end{equation}
By orthogonality of rows of the character table of the symmetric group, we have $\sum_{\mu_1 \dashv \frac{w}{k}}\frac{\chi^{R}(\mu_1)}{z_{\mu_1}} = 1$ when $R = (\frac{w}{k})$, and 0 otherwise. 

Hence, the coefficient of $\widehat{f}_{R}(q^k, Q^k)$ on the RHS of Equation \ref{eq:oc_split} is,

\[
\frac{-1}{k\left(q^{\frac{k}{2}} - q^{\frac{-k}{2}}\right)^2}
\]
when $\nu' = (\frac{w}{k})$, and 0 when $\nu' \neq (\frac{w}{k})$. This is in agreement with Equation \ref{eq:GV_fhat}.

\subsection{Vanishing of cross terms}

\label{sec:vanishing cross terms}

Suppose $d \geq 2$. For $1 \leq i \leq d$, fix $n_i \geq 1$ and representations $\nu_i'$. For fixed $b \leq w$ in the RHS of Equation \ref{eq:oc_split}, the following cross terms may appear,

\[
\prod_{i=1}^d \frac{1}{n_i}\widehat{f}_{\nu'_i}(q^{n_i}, Q^{n_i}) = \prod_{i=1}^b \prod_{t_i = 1}^{d_i}\frac{1}{n_{t_i}}\widehat{f}_{\nu'_{t_i}}(q^{n_{t_i}}, Q^{n_{t_i}})
\]
where we have partitioned the $(n_i$, $\widehat{f}_{\nu'_i})$ into buckets of size $d_i \geq 1$ with $d_1 + \ldots + d_b = d$. Each bucket $d_i$ contains  the tuples $(n_{t_i}$, $\widehat{f}_{\nu'_{t_i}})$ with $1 \leq t_i \leq d_i$. 

Suppose $\nu_{j_i}$ is a representation with $|\nu_{j_i}| = l_i \geq 1$. It suffices to consider the $Z^{\geq 2}_{\nu_{j_i}}$ terms in Equation \ref{eq:oc_split}. From Equation \ref{eq:Znu_dgeq2_simplified}, the coefficient of $\prod_{t_i = 1}^{d_i}\frac{1}{n_{t_i}}\widehat{f}_{\nu'_{t_i}}(q^{n_{t_i}}, Q^{n_{t_i}})$ in $Z^{d \geq 2}_{\nu_{j_i}}$ is obtained by fixing $d = d_i$, $n_{t_i} \geq 1$, and $\nu'_{t_i}$, then simplifying. Within each bucket $d_i$, the $\{\frac{1}{n_{t_i}} \widehat{f}_{\nu'_{t_i}}(q^{n_{t_i}}, Q^{n_{t_i}}) | 1 \leq t_i \leq d_i\}$ can be divided into $k_i$ groups of identical elements of size $d_j'$, for $1 \leq j \leq k_i$. Hence, we factor in $\binom{d_i}{d'_1 \ldots, d'_{k_i}}$ to account for the permutations in which $\prod_{t_i=1}^{d_i} \frac{1}{n_{t_i}} \widehat{f}_{\nu'_{t_i}}(q^{n_{t_i}}, Q^{n_{t_i}})$ arises. The coefficient of $\prod_{t_i = 1}^{d_i}\frac{1}{n_{t_i}}\widehat{f}_{\nu'_{t_i}}(q^{n_{t_i}}, Q^{n_{t_i}})$ in $Z_{\nu_{j_i}}$ is given by,

\begin{equation}
\label{eq:coeff_f_Znuji}
\begin{split}
    c(d_i, \{(n_{t_i}, \widehat{f}_{\nu'_{t_i}})\}, \nu_{j_i}) := \frac{1}{d_i!}\binom{d_i}{d'_1 \ldots, d'_{k_i}}\left(\prod_{t_i=1}^{d_i} \frac{1}{q^{\frac{n_{t_i}}{2}} - q^{\frac{-n_{t_i}}{2}}}\right)  \\ 
    \left[\sum_{\mu_1 \dashv |\nu'_1|} \frac{\chi^{\nu'_1}(\mu_1)}{z_{\mu_1}}P_{\mu_1}(q^{n_1})\ldots \sum_{\mu_{d_i} \dashv |\nu'_{d_i}|} \frac{\chi^{\nu'_{d_i}}(\mu_{d_i})}{z_{\mu_{d_i}}}P_{\mu_{d_i}}(q^{n_{d_i}}) \right.\\
    \left.\left[\sum_{\substack{\alpha'_{t_i} \dashv n_{t_i}|\nu'_{t_i}| \\ 1 \leq t_i \leq d_i}} \left(\prod_{t_i=1}^{d_i}\chi^{\alpha'_{t_i}}(n_{t_i}\mu_{t_i})\right)
    \left(\sum_{\substack{\alpha'_{12},\\ \ldots, \\ \alpha'_{12\ldots d_i-1}}}c^{\alpha'_{12}}_{\alpha'_1 \alpha'_2}\ldots c^{\nu}_{\alpha'_{12\ldots d_i-1} \alpha'_{d_i}}\right)\right]\right]
\end{split}
\end{equation}
Then, the coefficient of $\prod_{i=1}^d \frac{1}{n_i}f_{\nu'_i}(q^{n_i}, Q^{n_i}) = \prod_{i=1}^b \prod_{t_i = 1}^{d_i}\frac{1}{n_{t_i}}\widehat{f}_{\nu'_{t_i}}(q^{n_{t_i}}, Q^{n_{t_i}})$ is given by,

\begin{equation}
    \sum_{\substack{d_1 + \ldots d_b = d \\ d_i \geq 1 \\ l_i = \sum_{t_i=1}^{d_i} n_{t_i}|\nu'_{t_i}|}}\prod_{i=1}^b c(d_i, \{(n_{t_i}, \widehat{f}_{\nu'_{t_i}})\}, \nu_{j_i})
\end{equation}
where we sum over all possible buckets $(d_i)$. The condition $l_i = \sum_{t_i=1}^{d_i} n_{t_i}|\nu'_{t_i}|$ is required in the definition of $Z_{\nu_{j_i}}$.

Summing over all possible $b$, the coefficient of $\prod_{i=1}^d \frac{1}{n_i}f_{\nu'_i}(q^{n_i}, Q^{n_i})$ in Equation \ref{eq:oc_split} is given by, 

\begin{equation}
\label{eq:coeff_product}
    \sum_{b=1}^w\frac{(-1)^{b+1+w}}{b}\sum_{\substack{(l_1,\ldots, l_b) \in \mathbb{Z}_{\geq 1} \\ \sum_i l_i = w}}\sum_{\substack{d_1 + \ldots d_b = d \\ d_i \geq 1 \\ l_i = \sum_{t_i=1}^{d_i} n_{t_i}|\nu'_{t_i}|}}\prod_{i=1}^b\left[\sum_{j_i = 1}^{p(l_i)}c(d_i, \{(n_{t_i}, \widehat{f}_{\nu'_{t_i}})\}, \nu_{j_i})C_{\nu^T_{j_i}}\right]
\end{equation}

\begin{remark}
\label{rem:constraint}
    If $Z_{\nu_{j_i}}^{d_i \geq 2}$, defined in Equation \ref{eq:Znu_dgeq2_simplified}, is non-zero, by the properties of Littlewood-Richardson coefficients, we have $l_i = |\nu_{j_i}| = \sum_{t_i=1}^{d_i} |\alpha'_{t_i}| = \sum_{t_i=1}^{d_i} n_{t_i} |\nu_{t_i}'|$. Then, $w = \sum_{i=1}^b |\nu_{j_i}| = \sum_{i=1}^b \sum_{t_i=1}^{d_i} n_{t_i} |\nu_{t_i}'|$. Hence, for a cross-term $\prod_{i=1}^d \frac{1}{n_i} \widehat{f}_{\nu'_i}(q^{n_i}, Q^{n_i})$ appearing in the RHS of Equation \ref{eq:oc_split}, we have $\sum_{i=1}^d n_i |\nu'_i| = w$.
\end{remark}
Switching the sums of $(d_i)$ and $(l_i)$, Equation \ref{eq:coeff_product} is equivalently, 

\begin{equation}
\label{eq:coeff_product2}
    \sum_{b=1}^w\frac{(-1)^{b+1+w}}{b}\sum_{\substack{d_1 + \ldots d_b = d \\ d_i \geq 1}}\sum_{\substack{(l_1,\ldots, l_b) \in \mathbb{Z}_{\geq 1} \\ l_1 + \ldots + l_b = w \\ l_i = \sum_{t_i=1}^{d_i} n_{t_i}|\nu'_{t_i}|}}\prod_{i=1}^b \left[\sum_{j_i = 1}^{p(l_i)}c(d_i, \{(n_{t_i}, \widehat{f}_{\nu'_{t_i}})\}, \nu_{j_i})C_{\nu^T_{j_i}}\right]
\end{equation}

\subsubsection{Vanishing}

We show that Equation \ref{eq:coeff_product2} vanishes. Using the definition of \\$c(d_i, \{(n_{t_i}, \widehat{f}_{\nu'_{t_i}})\}, \nu_{j_i})$ in Equation \ref{eq:coeff_f_Znuji}, Equation \ref{eq:coeff_product2} becomes,

\begin{equation}
\label{eq:vanishing_exp}
\begin{split}
    \sum_{b=1}^w\frac{(-1)^{b+1+w}}{b}\sum_{\substack{d_1 + \ldots d_b = d \\ d_i \geq 1}}\sum_{\substack{(l_1,\ldots, l_b) \in \mathbb{Z}_{\geq 1} \\ l_1 + \ldots + l_b = w \\ l_i = \sum_{t_i = 1}^{d_i}n_{t_i} |\nu'_{t_i}|}}\prod_{i=1}^b \left[\sum_{j_i = 1}^{p(l_i)}\frac{1}{d_i!}\binom{d_i}{d'_1 \ldots, d'_{k_i}}\left(\prod_{t_i=1}^{d_i}\frac{1}{q^{\frac{n_{t_i}}{2}} - q^{\frac{-n_{t_i}}{2}}}\right)\right.\\
    \left.\left[\sum_{\mu_1 \dashv |\nu'_1|} \frac{\chi^{\nu'_1}(\mu_1)}{z_{\mu_1}}P_{\mu_1}(q^{n_1}) \ldots \sum_{\mu_{d_i} \dashv |\nu'_{d_i}|} \frac{\chi^{\nu'_{d_i}}(\mu_{d_i})}{z_{\mu_{d_i}}}P_{\mu_{d_i}}(q^{n_{d_i}}) \right.\right.\\
    \left.\left.\left[\sum_{\substack{\alpha'_{t_i} \dashv n_{t_i}|\nu'_{t_i}| \\ 1 \leq t_i \leq d_i}} \left(\prod_{t_i=1}^{d_i}\chi^{\alpha'_{t_i}}(n_{t_i}\mu_{t_i})\right)
    \left(\sum_{\substack{\alpha'_{12},\\ \ldots, \\ \alpha'_{12\ldots d_i-1}}}c^{\alpha'_{12}}_{\alpha'_1 \alpha'_2}\ldots c^{\nu}_{\alpha'_{12\ldots d_i-1} \alpha'_{d_i}}\right)\right]\right]C_{\nu^T_{j_i}}\right]
\end{split}
\end{equation}
Recall that $C_{\nu^T_{j_i}} = s_{\nu^T_{j_i}}(q^{\rho})$. Since $c_{\alpha'_x \alpha'_y}^{\alpha'_{xy}} = c_{\alpha'^T_x \alpha'^T_y}^{\alpha'^T_{xy}}$ for all $x, y$, and Schur functions form a $\mathbb{Z}$-basis of symmetric functions, we have,

\[
\sum_{j_i=1}^{p(l_i)} \left(\sum_{\substack{\alpha'_{12},\\ \ldots, \\ \alpha'_{12\ldots d_i-1}}}c^{\alpha'^T_{12}}_{\alpha'^T_1 \alpha'^T_2}\ldots c^{\nu^T_{j_i}}_{\alpha'^T_{12\ldots d_i-1} \alpha'^T_{d_i}}\right)s_{\nu^T_{j_i}}(q^{\rho}) = s_{\alpha'^T_1}\ldots s_{\alpha'^T_{d_i}}(q^{\rho})
\]
Moving the sum over $j_i$ in, Equation \ref{eq:vanishing_exp} becomes,

\begin{equation}
\label{eq:vanishing_exp_1}
\begin{split}
    \sum_{b=1}^w\frac{(-1)^{b+1+w}}{b}\sum_{\substack{d_1 + \ldots d_b = d \\ d_i \geq 1}}\sum_{\substack{(l_1,\ldots, l_b) \in \mathbb{Z}_{\geq 1} \\ l_1 + \ldots + l_b = w \\ l_i = \sum_{t_i = 1}^{d_i}n_{t_i} |\nu'_{t_i}|}}\prod_{i=1}^b \left[\frac{1}{d_i!}\binom{d_i}{d'_1 \ldots, d'_{k_i}}\left(\prod_{t_i=1}^{d_i}\frac{1}{q^{\frac{n_{t_i}}{2}} - q^{\frac{-n_{t_i}}{2}}}\right)\right.\\
    \left.\left[\sum_{\mu_1 \dashv |\nu'_1|} \frac{\chi^{\nu'_1}(\mu_1)}{z_{\mu_1}}P_{\mu_1}(q^{n_1})\ldots \sum_{\mu_{d_i} \dashv |\nu'_{d_i}|} \frac{\chi^{\nu'_{d_i}}(\mu_{d_i})}{z_{\mu_{d_i}}}P_{\mu_{d_i}}(q^{n_{d_i}}) \right.\right. \\
   \left. \left.\left[\sum_{\substack{\alpha'_{t_i} \dashv n_{t_i}|\nu'_{t_i}| \\ 1 \leq t_i \leq d_i}} \left(\prod_{t_i=1}^{d_i}\chi^{\alpha'_{t_i}}(n_{t_i}\mu_{t_i})\right)s_{\alpha'^T_1}\ldots s_{\alpha'^T_{d_i}}(q^{\rho})\right]\right]\right]
\end{split}
\end{equation}
For $1 \leq t_i \leq d_i$, the change of basis formula gives us,

\begin{align*}
    \sum_{\alpha_{t_i}' \dashv n_{t_i} |\nu'_{t_i}|} \chi^{\alpha'_{t_i}}(n_{t_i} \mu_{t_i}) s_{\alpha'^T_{t_i}}(q^{\rho}) &= \sum_{\alpha_{t_i}' \dashv n_{t_i} |\nu'_{t_i}|}  (-1)^{|n_{t_i} \mu_{t_i}| - l(n_{t_i} \mu_{t_i})}\chi^{\alpha'^T_{t_i}}(n_{t_i} \mu_{t_i})s_{\alpha'^T_{t_i}}(q^{\rho}) \\
    &=  (-1)^{|n_{t_i} \mu_{t_i}| - l(n_{t_i} \mu_{t_i})}p_{n_{t_i} \mu_{t_i}}(q^{\rho})
\end{align*}
Equation \ref{eq:vanishing_exp_1} becomes, 

\begin{equation}
\label{eq:vanishing_exp_2}
\begin{split}
    \sum_{b=1}^w\frac{(-1)^{b+1+w}}{b}\sum_{\substack{d_1 + \ldots d_b = d \\ d_i \geq 1}}\sum_{\substack{(l_1,\ldots, l_b) \in \mathbb{Z}_{\geq 1} \\ l_1 + \ldots + l_b = w \\ l_i = \sum_{t_i = 1}^{d_i}n_{t_i} |\nu'_{t_i}|}}\prod_{i=1}^b \left[\frac{1}{d_i!}\binom{d_i}{d'_1 \ldots, d'_{k_i}}\left(\prod_{t_i=1}^{d_i}\frac{1}{q^{\frac{n_{t_i}}{2}} - q^{\frac{-n_{t_i}}{2}}}\right)\right.\\
    \left.\left[\sum_{\mu_1 \dashv |\nu'_1|} \frac{\chi^{\nu'_1}(\mu_1)}{z_{\mu_1}}P_{\mu_1}(q^{n_1}) \ldots \sum_{\mu_{d_i} \dashv |\nu'_{d_i}|} \frac{\chi^{\nu'_{d_i}}(\mu_{d_i})}{z_{\mu_{d_i}}}P_{\mu_{d_i}}(q^{n_{d_i}}) \right.\right.\\
    \left.\left.\left[\prod_{t_i=1}^{d_i}(-1)^{|n_{t_i} \mu_{t_i}| - l(n_{t_i} \mu_{t_i})}p_{n_{t_i} \mu_{t_i}}(q^{\rho}))\right]\right]\right]
\end{split}
\end{equation}
We have $\prod_{i=1}^b\prod_{t_i=1}^{d_i} (-1)^{|n_{t_i}\mu_{t_i}|} = (-1)^w$. By Equation \ref{eq:p and P}, Equation \ref{eq:vanishing_exp_2} becomes,

\begin{equation}
\label{eq:vanishing_exp_3}
\begin{split}
    \sum_{b=1}^w\frac{(-1)^{b+1}}{b}\sum_{\substack{d_1 + \ldots d_b = d \\ d_i \geq 1}}\sum_{\substack{(l_1,\ldots, l_b) \in \mathbb{Z}_{\geq 1} \\ l_1 + \ldots + l_b = w \\ l_i = \sum_{t_i = 1}^{d_i}n_{t_i} |\nu'_{t_i}|}}\prod_{i=1}^b \left[\frac{1}{d_i!}\binom{d_i}{d'_1 \ldots, d'_{k_i}}\left(\prod_{t_i=1}^{d_i}\frac{-1}{\left(q^{\frac{n_{t_i}}{2}} - q^{\frac{-n_{t_i}}{2}}\right)^2}\right)\right.\\
    \left.\left[\sum_{\mu_1 \dashv \nu'_1} \frac{\chi^{\nu'_1}(\mu_1)}{z_{\mu_1}} \ldots \sum_{\mu_{d_i} \dashv |\nu'_{d_i}|} \frac{\chi^{\nu'_{d_i}}(\mu_{d_i})}{z_{\mu_{d_i}}}\right]\right]
\end{split}
\end{equation}
By orthogonality of the rows of the character table of the symmetric group, we have that for all $1 \leq t_i \leq d_i$, $\sum_{\mu_{t_i} \dashv |\nu'_{t_i}|}\frac{\chi^{\nu'_{t_i}}(\mu_{t_i})}{z_{\mu_{t_i}}} = 1$ if $\nu'_{t_i}$ is the representation with a single row of length $(|\nu'_{t_i}|)$, and 0 otherwise. 

\begin{remark}
\label{rem:vanishing_other}
    This proves vanishing for products $\prod_{i=1}^d \frac{1}{n_i}\widehat{f}_{\nu'_i}(q^{n_i}, Q^{n_i})$ in which $\exists i \geq 1$ such that $\nu'_i$ is not the single row.
\end{remark}
Henceforth, it suffices to let $\nu_i'$ be single row representations for all $i$. Equation \ref{eq:vanishing_exp_3} simplifies to,

\begin{equation}
\label{eq:vanishing_exp_4}
\begin{split}
    (-1)^{d} \left(\prod_{i=1}^d \frac{1}{\left(q^{\frac{n_{i}}{2}} - q^{\frac{-n_{i}}{2}}\right)^{2}}\right) \sum_{b=1}^w\frac{(-1)^{b+1}}{b}\sum_{\substack{d_1 + \ldots d_b = d \\ d_i \geq 1}}\sum_{\substack{(l_1,\ldots, l_b) \in \mathbb{Z}_{\geq 1} \\ l_1 + \ldots + l_b = w \\ l_i = \sum_{t_i = 1}^{d_i}n_{t_i} |\nu'_{t_i}|}}\prod_{i=1}^b \frac{1}{d'_1! \ldots d'_{k_i}!}
\end{split}
\end{equation}
Canceling constants, it suffices to show the vanishing of,

\begin{equation}
\label{eq:exp}
   \sum_{b=1}^w\frac{(-1)^{b+1}}{b}\sum_{\substack{d_1 + \ldots d_b = d \\ d_i \geq 1}}\sum_{\substack{(l_1,\ldots, l_b) \in \mathbb{Z}_{\geq 1} \\ l_1 + \ldots + l_b = w \\ l_i = \sum_{t_i = 1}^{d_i} n_{t_i}|\nu'_{t_i}|}}\prod_{i=1}^b \frac{1}{d'_1! \ldots d'_{k_i}!} = 0 
\end{equation}

\begin{remark}
    In the first grouping into buckets $d_i$ in the above, we treat all $\widehat{f}_{\nu'_i}(q^{n_i}, Q^{n_i})$ as distinct, labeled elements, even though there may exists $i \neq j$ such that $\nu'_i = \nu'_j$. After partitioning into the buckets $d_i$, the elements are then considered unlabeled, and we therefore divide by $d'_1! \ldots d'_{k_i}!$ to get the correct number of permutations. We have $\sum_{j=1}^{k_i} d'_j = d_i$ for all $i$. 
\end{remark}

\begin{proposition}
\label{prop:stirling}
Equation \ref{eq:exp} vanishes.
\begin{proof}

First, we suppose that $n_{t_i} = 1$ in Equation \ref{eq:exp} for all $1 \leq t_i \leq d_i$ and $1 \leq i \leq d$. We show the vanishing of,

\[
\sum_{b=1}^w\frac{(-1)^{b+1}}{b}\sum_{\substack{d_1 + \ldots d_b = d \\ d_i \geq 1}}\sum_{\substack{(l_1,\ldots, l_b) \in \mathbb{Z}_{\geq 1} \\ l_1 + \ldots + l_b = w \\ l_i = \sum_{t_i = 1}^{d_i}|\nu'_{t_i}|}}\prod_{i=1}^b \frac{1}{d'_1! \ldots d'_{k_i}!} = 0
\]
Note that the sum over $b$ only goes up to $d$ since $d_i \geq 1$. Recall for the product $\prod_{i=1}^d \widehat{f}_{\nu'_i}(q, Q)$, it's required that $\sum_i |\nu'_i| = w$ by Remark \ref{rem:constraint}. Hence, we can remove the sum over $(l_i)$. It suffices to show,

\begin{equation}
\label{eq:before stirling}
    \sum_{b=1}^d\frac{(-1)^{b+1}}{b}\sum_{\substack{d_1 + \ldots d_b = d \\ d_i \geq 1}}\prod_{i=1}^b \frac{1}{d'_1 ! \ldots d'_{k_i}!} = 0  
\end{equation}
We argue that,

\begin{equation}
\label{eq:Stirling}
   \sum_{\substack{d_1 + \ldots d_b = d \\ d_i \geq 1}}\prod_{i=1}^b \frac{1}{d'_1 ! \ldots d'_{k_i}!} = \frac{b! S(d, b)}{d!}
\end{equation}
where $S(d,b)$ is the second Stirling number with $d$ items and $b$ buckets. \footnote{The second Stirling number $S(d, b)$ is the number of ways to partition $d$ distinct, labeled objects into $b$ non-empty sets with unlabeled objects.} Recall that $(e^x-1)^b := \sum_{n=b}^{\infty} \frac{b!S(n,b)}{n!}x^n$. 

Let $M$ be the number of distinct integers in $\{|\nu'_j| | 1 \leq j \leq d\}$. Let $x_1, x_2,\ldots, x_M$ be formal variables. Label the distinct integers $|\nu'_{x_i}|$ for $1 \leq i \leq M$. Let $x := x_1 + \ldots + x_M$. Consider $x^d = x_1^{a_1}\ldots x_M^{a_M}$, where $d = a_1 + \ldots + a_M$. We interpret $x_i^{a_i}$ is the number of occurrences of $|\nu'_{x_i}|$. We set $a_1 + \ldots + a_M = d$ since the product $\prod_{i=1}^d \widehat{f}_{\nu'_i}(q, Q)$ has $d$ terms.  

Consider $(e^x - 1)^b = (e^{x_1+\ldots + x_M} - 1)^b$. We have $e^x = \sum_{a_i \geq 0} \frac{x_1^{a_1}\ldots x_M^{a_M}}{a_1! \ldots a_M!}$. We see that the coefficient of $x^d$ in $(e^{x} - 1)^b$ is equal to $\sum_{\substack{d_1 + \ldots d_b = d \\ d_i \geq 1}}\prod_{i=1}^b \frac{1}{d'_1 ! \ldots d'_{k_i}!}$, with $d'_{k_j}$ for $1 \leq j \leq i$ is defined in Section \ref{sec:vanishing cross terms}, and $\sum_{j=1}^{k_i} d'_j = d_i$ for all $i$. By definition, the coefficient of $x^d$ is also $\frac{b! S(d,b)}{d!}$. 

Therefore, it suffices to show that for $d \geq 2$, 

\[
\sum_{b=1}^d \frac{(-1)^{b+1}}{b} \frac{b! S(d,b)}{d!} = 0
\]
This follows from,

\begin{align*}
    \sum_{d=0}^{\infty} \left( \sum_{b=1}^d \frac{(-1)^{b+1}}{b} \frac{b!S(d,b)}{d!}\right) x^d &= \sum_{b=1}^{\infty} \frac{(-1)^{b+1}}{b}\sum_{d=b}^{\infty} \frac{b! S(d,b)}{d!}x^d \\
    &= \sum_{b=1}^{\infty} \frac{(-1)^{b+1}}{b} b! \frac{(e^x - 1)^b}{b!} \\
    &= \sum_{b=1}^{\infty} \frac{(-1)^{b+1}}{b}(e^x - 1)^b \\
    &= \log(1 + e^x - 1) \\
    &= x
\end{align*}
Thus, when $n_i = 1$ and the $\nu'_i$ are single rows, the coefficient of $\prod_{i=1}^d \widehat{f}_{\nu'_i}(q, Q)$ in Equation \ref{eq:oc_split} vanishes. The vanishing for $\nu'_i$ when it is not a single row is explained in Remark \ref{rem:vanishing_other}. For $n_{i} \geq 1$, the vanishing of Equation \ref{eq:exp} follows from the derivation for vanishing when $n_i = 1$, because we consider tuples $(n_{i}, f_{\nu'_{i}})$, in which each $\nu'_{i}$ is paired with an $n_{i}$.
\end{proof}
\end{proposition}

\begin{proof}[Proof of Theorem \ref{thm:oc}]

\label{pf:oc}
    By plugging in Theorem \ref{thm:oc} into the Gopakumar-Vafa formula in Equation \ref{eq:GV}, we derived Equation \ref{eq:oc_split}.
    
    In Section \ref{sec:matching}, we matched the coefficients of $\widehat{f}_{R}(q^k, Q^k)$ on the LHS and RHS of Equation \ref{eq:oc_split}. In Section \ref{sec:vanishing cross terms}, we found the coefficients of cross terms $\prod_{i=1}^d \frac{1}{n_i}\widehat{f}_{\nu'_i}(q^{n_i}, Q^{n_i})$ for $d \geq 2$ appearing in the RHS of Equation \ref{eq:oc_split}, and showed their vanishing in Proposition \ref{prop:stirling}. Since $n_i \geq 1$, this also means vanishing for the coefficients of $\prod_{i=1}^d \widehat{f}_{\nu'_i}(q^{n_i}, Q^{n_i})$. 

    Thus, Equation \ref{eq:oc_split} is an equality, which holds on the level of generating functions. It implies Theorem \ref{thm:oc}, or equality on the level of individual invariants, 

    \[
    n_g(K_{\widehat{X}}, \pi^*\beta-wC) = (-1)^g N^{LMOV}_{g, (w)}(K_X/L, \beta)
    \]
\end{proof}

\section{Examples} 

We provide examples of Theorem \ref{thm:GRZZ} in $w = 1,2,3$. We write $\widehat{f}_R$, in place of $\widehat{f}_R(q, Q)$ for simplicity.
\label{sec:examples}
\begin{example}[Winding-1]
\label{ex:w1}

Setting $w = 1$ in Theorem \ref{thm:GRZZ}, we have,

\begin{align*}
  Q^{-C}\sum_{\substack{g \geq 0, \\ \beta \in NE(X)}}N_g(K_{\widehat{X}}, \pi^*\beta - C)Q^{\pi^*\beta}\hbar^{2g-2} &= -Z_{(1)}C_{(1)\varnothing \varnothing}Q^E \\
  &= -(q^{\frac{1}{2}} - q^{\frac{-1}{2}})^{-2}\widehat{f}_{(1)}Q^E
\end{align*}
The Gopakumar-Vafa formula \cite{GV1} \cite{GV2} for the curve class $\pi^*\beta - C$ is,

\begin{equation*}
    \begin{split}
        \sum_{\substack{g \geq 0, \\ \beta \in NE(X)}}N_g(K_{\widehat{X}}, \pi^*\beta - C)Q^{\pi^*\beta - C}\hbar^{2g-2} = \sum_{\substack{g \geq 0, \\ \beta \in NE(X)}}(-1)^{g-1}n_g(K_{\widehat{X}}, \pi^*\beta - C) Q^{\pi^*\beta - C}\\(q^{\frac{1}{2}} - q^{\frac{-1}{2}})^{2g-2}
    \end{split}
\end{equation*}
where $q = e^{i \hbar}$. Hence multiplying by $-(q^{\frac{1}{2}} - q^{\frac{-1}{2}})^{2}$ and canceling $Q^{-C}$ with $Q^E$, we have

\begin{equation}
\label{eq:winding-1_eq}
   \sum_{\substack{g \geq 0, \\ \beta \in NE(X)}}(-1)^{g}n_g(K_{\widehat{X}}, \pi^*\beta - C) Q^{\pi^*\beta}(q^{\frac{1}{2}} - q^{\frac{-1}{2}})^{2g} = \widehat{f}_{(1)} 
\end{equation}
under the variable change $z = (q^{\frac{1}{2}} - q^{\frac{-1}{2}})^{2}$. Hence, we have,

\begin{corollary}
  \label{cor:winding-1}
    For $w = 1$, Example \ref{ex:w1} of Theorem \ref{thm:GRZZ} implies,
    \begin{align*}
      n_g(K_{\widehat{X}}, \pi^*\beta-C) &= (-1)^{g}N^{\mathrm{LMOV}}_{g, (1)}(K_{X}/L, \beta)
    \end{align*}
\end{corollary}

\end{example}

\begin{example}[Theorem \ref{thm:oc} in winding-2]
\label{ex:w2}
Setting $w = 2$ in Equation \ref{thm:GRZZ}, we have 

\[
Q^{-2C}\sum_{\substack{g \geq 0, \\ \beta \in NE(X)}}N_g(K_{\widehat{X}}, \pi^*\beta - 2C)Q^{\pi^*\beta} \hbar^{2g-2}= \left( Z_{(2)} C_{(1,1)} + Z_{(1,1)} C_{(2)} - \frac{\left(Z_{(1)} C_{(1)}\right)^2}{2}\right)Q^{2E} 
\]
The relevant quantities are described in \cite{AKMV} and are,
\begin{align*}
  C_{(2)} &= \frac{q^2}{(q-1)(q^2-1)} \\
  C_{(1,1)} &= \frac{q}{(q-1)(q^2-1)} \\
  C_{(1)} &= (q^{\frac{1}{2}} - q^{\frac{-1}{2}})^{-1} \\
  Z_{(2)} &= f_{(2)} + \frac{1}{2}f_{(1)}(q^2, Q^2) + \frac{1}{2}f_{(1)}(q, Q)^2  \\
  Z_{(1,1)} &= f_{(1,1)} - \frac{1}{2}f_{(1)}(q^2, Q^2) + \frac{1}{2}f_{(1)}(q, Q)^2 \\
  f_{(2)} &= (q^{\frac{1}{2}} - q^{\frac{-1}{2}})^{-1}(q^{\frac{-1}{2}}\widehat{f}_{(2)} - q^{\frac{1}{2}}\widehat{f}_{(1,1)}) \\
  f_{(1,1)} &= (q^{\frac{1}{2}} - q^{\frac{-1}{2}})^{-1}(-q^{\frac{1}{2}}\widehat{f}_{(2)} + q^{\frac{-1}{2}}\widehat{f}_{(1,1)}) \\
  Z_{(1)} &= f_{(1)} = (q^{\frac{1}{2}} - q^{\frac{-1}{2}})^{-1}\widehat{f}_{(1)}
\end{align*}
where $\widehat{f}$ is defined in Equation \ref{eq:def_fhat}. We have $C_{(1,1)} + C_{(2)} = (q^{\frac{1}{2}} - q^{\frac{-1}{2}})^{-2}$ and $C_{(1,1)} - C_{(2)} = (q^{-1}-q)^{-1}$. The Gopakumar-Vafa formula for curve class $\pi^*\beta - 2C$ is,

\begin{equation}
\label{eq:winding-2_GV}
 \begin{split}
   \sum_{\substack{g \geq 0, \\ \beta \in NE(X)}}N_g(K_{\widehat{X}}, \pi^*\beta - 2C) \hbar^{2g-2} Q^{\pi^*\beta - 2C} = \sum_{\substack{g \geq 0, \\ \beta \in NE(X)}}(-1)^{g-1}\left[n_g(K_{\widehat{X}}, \pi^*\beta - 2C)\right. \\ \left.\left(q^{\frac{1}{2}} - q^{\frac{-1}{2}}\right)^{2g-2} + \frac{1}{2}n_g\left(K_{\widehat{X}}, \frac{\pi^*\beta}{2} - C\right)(q-q^{-1})^{2g-2}\right]Q^{\pi^*\beta - 2C}
\end{split}   
\end{equation}
Plugging in expressions, we have,

\begin{equation}
\label{eq:winding2_vertex}
    \begin{split}
        Z_{(2)} C_{(1,1)} + Z_{(1,1)} C_{(2)} - \frac{1}{2}(Z_{(1)} C_{(1)})^2 &= f_{(2)} \frac{q}{(q-1)(q^2-1)} + f_{(1,1)}\frac{q^2}{(q-1)(q^2-1)} \\
 &- \frac{1}{2(q-q^{-1})}f_{(1)}(q^2, Q^2) \\
 &= \left(q^{\frac{-1}{2}}\widehat{f}_{(2)} - q^{\frac{1}{2}}\widehat{f}_{(1,1)}\right)\frac{q}{(q^{\frac{1}{2}} - q^{\frac{-1}{2}})(q-1)(q^2-1)} \\
 &+ \left(-q^{\frac{1}{2}}\widehat{f}_{(2)} + q^{\frac{-1}{2}}\widehat{f}_{(1,1)}\right)\frac{q^2}{(q^{\frac{1}{2}} - q^{\frac{-1}{2}})(q-1)(q^2-1)} \\
 &- \frac{1}{2(q-q^{-1})^2}\widehat{f}_{(1)}(q^2, Q^2) \\
 &= \frac{-1}{(q^{\frac{1}{2}} - q^{\frac{-1}{2}})^2}\widehat{f}_{(2)}(q, Q) - \frac{1}{2(q-q^{-1})^2}\widehat{f}_{(1)}(q^2, Q^2)  
    \end{split}
\end{equation}
Setting Equation \ref{eq:winding-2_GV} equal to Equation \ref{eq:winding2_vertex} implies, on the level of individual invariants,

\begin{corollary}
\label{cor:winding-2}
For $w = 2$, Example \ref{ex:w2} of Theorem \ref{thm:GRZZ} implies,

\begin{align*}
n_g(K_{\widehat{X}}, \pi^*\beta - 2C) &= (-1)^g N^{\mathrm{LMOV}}_{g, (2)}(K_X/L, \beta) 
\end{align*} 
\end{corollary}
\end{example}

\begin{example}[Winding-3]
\label{ex:w3}
Setting $w = 3$ in Theorem \ref{thm:GRZZ}, we have, 

\begin{equation}
\label{eq:w3_thm}
 \begin{split}
    \sum_{g, \beta}N_g(K_{\widehat{X}}, \pi^*\beta - 3C)\hbar^{2g-2}Q^{\pi^*\beta-3C} &= \left( -Z_{(3)} C_{(1,1,1)} - Z_{(2,1)}C_{(2,1)} - Z_{(1,1,1)}C_{(3)}\right. \\
    &+ \left. Z_1 C_1\left(Z_{(2)}C_{(1,1)} 
    + Z_{(1,1)}C_{(2)}\right) - \frac{\left(Z_1 C_1\right)^3}{3}\right) Q^{3E}   
 \end{split}
\end{equation}
The Gopakumar-Vafa formula for curve class $\pi^*\beta - 3C$ is,

\begin{equation}
\label{eq:winding-3_GV}
 \begin{split}
    \sum_{\substack{g \geq 0, \\ \beta \in NE(X)}}N_g(K_{\widehat{X}}, \pi^*\beta - 3C) \hbar^{2g-2}Q^{\pi^*\beta - 3C} = \sum_{\substack{g \geq 0, \\ \beta \in NE(X)}}(-1)^{g-1}\left[n_g(K_{\widehat{X}}, \pi^*\beta - 3C)\right.\\
     \left.(q^{\frac{1}{2}} - q^{\frac{-1}{2}})^{2g-2} + \frac{1}{3}n_g\left(K_{\widehat{X}}, \frac{\pi^*\beta}{3} - C\right)(q^{\frac{3}{2}}-q^{\frac{-3}{2}})^{2g-2}\right]Q^{\pi^*\beta - 3C}
\end{split}   
\end{equation}

We have,

\begin{align*}
Z_3 &= (q^{\frac{1}{2}} - q^{\frac{-1}{2}})^{-1}(q^{-1}\widehat{f}_3 - \widehat{f}_{(2,1)} + q\widehat{f}_{(1,1,1)}) + (q^{\frac{1}{2}} - q^{\frac{-1}{2}})^{-2}\widehat{f}_1(q^{\frac{-1}{2}}\widehat{f}_2 - q^{\frac{1}{2}}\widehat{f}_{(1,1)}) \\
&+ \frac{1}{6}(q^{\frac{1}{2}} - q^{\frac{-1}{2}})^{-3}\widehat{f}_1^3 + \frac{1}{2}(q^{\frac{1}{2}} - q^{\frac{-1}{2}})^{-1}\widehat{f}_1 (q-q^{-1})^{-1}\widehat{f}_1(q^2, Q^2) \\
&+ \frac{1}{3}(q^{\frac{3}{2}} - q^{\frac{-3}{2}})^{-1}\widehat{f}_1(q^3, Q^3) \\
Z_{(2,1)} &= (q^{\frac{1}{2}} - q^{\frac{-1}{2}})^{-1}(-\widehat{f}_3 + (q^{-1} - 1 + q)\widehat{f}_{(2,1)} - \widehat{f}_{(1,1,1)}) \\
&+ (q^{\frac{1}{2}} - q^{\frac{-1}{2}})^{-2}\widehat{f}_1 \left[(q^{\frac{-1}{2}} - q^{\frac{1}{2}})\widehat{f}_2 + (q^{\frac{-1}{2}} - q^{\frac{1}{2}})\widehat{f}_{11})\right] \\
&+ \frac{1}{3}(q^{\frac{1}{2}} - q^{\frac{-1}{2}})^{-3}\widehat{f}_1^3 - \frac{1}{3}(q^{\frac{3}{2}} - q^{\frac{-3}{2}})^{-1}\widehat{f}_1(q^3, Q^3) \\
Z_{(1,1,1)} &= (q^{\frac{1}{2}} - q^{\frac{-1}{2}})^{-1}(q\widehat{f}_3 - \widehat{f}_{(2,1)} + q^{-1}\widehat{f}_{(1,1,1)}) + (q^{\frac{1}{2}} - q^{\frac{-1}{2}})^{-2}(-q^{\frac{1}{2}}\widehat{f}_2 + q^{\frac{-1}{2}}\widehat{f}_{(1,1)})\widehat{f}_1 \\
&+ \frac{1}{6}(q^{\frac{1}{2}} - q^{\frac{-1}{2}})^{-3}\widehat{f}_1^3 - \frac{1}{2} (q^{\frac{1}{2}} - q^{\frac{-1}{2}})^{-1} (q - q^{-1})^{-1}\widehat{f}_1 \widehat{f}_1(q^2, Q^2) \\
&+ \frac{1}{3}(q^{\frac{3}{2}} - q^{\frac{-3}{2}})^{-1}\widehat{f}_1(q^3, Q^3)
\end{align*}
The vertex functions in winding-3 are,

\begin{align*}
    C_{(3)} &= \frac{q^{9/2}}{(q-1)(q^2 - 1)(q^3 - 1)} \\
    C_{(2, 1)} &= \frac{q^{5/2}}{(q-1)^2(q^3 - 1)} \\
    C_{(1, 1, 1)} &= \frac{q^{3/2}}{(q-1)(q^2 - 1)(q^3 - 1)} 
\end{align*}
\end{example}
Other relevant quantities are given in Example \ref{ex:w2}.

Comparing with Equation \ref{eq:winding-3_GV}, we expect that after simplifying the right side of Equation \ref{eq:w3_thm}, the result should be, 

\[
\frac{-1}{(q^{\frac{1}{2}} - q^{\frac{-1}{2}})^{2}}\widehat{f}_3(q, Q) - \frac{1}{3(q^{\frac{3}{2}} - q^{\frac{-3}{2}})^{2}}\widehat{f}_1(q^3, Q^3)
\]
Indeed, it is true. Hence, after equating the resulting generating functions, then on the level of individual invariants, we have, 

\begin{corollary}
\label{cor:winding-3}
For $w = 3$, Example \ref{ex:w3} of Theorem \ref{thm:GRZZ} implies,

\begin{align*}
n_g(K_{\widehat{X}}, \pi^*\beta - 3C) &= (-1)^g N^{\mathrm{LMOV}}_{g, (3)}(K_X/L, \beta) 
\end{align*} 
\end{corollary}

\bibliographystyle{alpha}
\bibliography{refs}

\end{document}